\documentclass[10pt]{amsart}
\usepackage{amsmath,amssymb,euscript,amsthm}

\title{On the topological essential
range and regularity of cocycles over compact and generic systems}
\author{Vyacheslav Kulagin}
\address{Vyacheslav Kulagin, Institute for Low Temperature Physics and Engineering, 47 Lenin Ave., Kharkov 61103, Ukraine.}
\email{\{skulagin@rambler.ru\}} 
\thanks{The author was supported by grant INTAS YSF-05-109-5200.}
\newtheorem{theorem}{\sc Theorem}[section]

\newtheorem{proposition}[theorem]{\sc Proposition}
\newtheorem{corollary}[theorem]{\sc Corollary}
\newtheorem{definition}[theorem]{Definition}

\theoremstyle{remark}
\newtheorem{remark}[theorem]{\sc Remark}
\def\N{\mathbb{N}}
\def\Z{\mathbb{Z}}
\def\T{\mathbb{T}}
\def\mR{\mathbb{R}}

\def\R{\EuScript{R}}

\newcommand{\wt}[1]{\widetilde{#1}}
\newcommand{\wh}[1]{\widehat{#1}}

\newcommand{\ov}[1]{\overline{#1}}
\begin{document}

\maketitle

\begin{abstract}
\noindent We consider the notions of topological essential range and
regularity for continuous cocycles over minimal $\Z$-systems introduced in
\cite{GH} and discuss relations with their generic counterparts. The
alternative generic definitions can be given by using the notion of generic
Mackey action associated with a cocycle. We further present a description
of recurrent cocycles over minimal rotations with values in discrete groups
and derive several consequences.
\end{abstract}

\section{Introduction and preliminaries}

The concept of topological essential range for continuous cocycles over
compact minimal systems was introduced in \cite{Atkinson},\cite{LeM} where
the case of abelian groups of values has been under research. A
generalization to a non-abelian case was made in \cite{GH}. It was applied
to obtain several results on regularity of topological cocycles which
permit to describe the topological ergodic decomposition and structure of
orbit closures in skew product actions.

Our aim is to describe the interplay between the above mentioned notions
and their generic analogues. We consider %the corresponding notion for
Polish group valued cocycles over Polish minimal systems and suggest a
parallel approach based on the notion of generic Mackey action associated
with a cocycle (\cite{GK}). It appears that several important properties of
continuous cocycles related with essential ranges and regularity can be
derived using such a generic approach but in more general situations.

We further proceed with the study of regularity problem for cocycles over a
minimal rotation on a compact monothetic group. We completely describe the
case when the group of values is discrete. It is applied for getting
several reduction results for cocycles taking values in locally compact
groups. Unlike the situation in \cite{GH}, where rotations on locally
connected compact groups are considered, the most of our results are
referred to a disconnected or arbitrary base space.

Let $X$ be a perfect Polish space, $T$ a homeomorphism of $X$ and $G$ a
Polish group. A continuous map $f:X\to G$ defines a $\Z$-cocycle by:
\begin{equation*}
f(n,x)=
\begin{cases}
f(T^{n-1} x)\cdot \ldots \cdot f(T x)\cdot f(x) & \textup{if}\enspace n\geq 1, \\
e & \textup{if}\enspace n=0, \\
f(-n,T^n x)^{-1} & \textup{if}\enspace n < 0,
\end{cases}
\end{equation*}

A skew product action is a continuous $\Z$-action on $X\times G$ defined
by: \begin{equation*} T_f^n (x,g) = \big(T^n x, f(n,x)\cdot g\big)
\end{equation*}

A cocycle $f$ is said to be {\sl recurrent} if the skew product $T_f$ is
topologically conservative (i.e. for every open nonempty $O\subset X\times
G$ there exists an integer $n\neq 0$ with $T_f ^n O\cap O \neq
\varnothing$). $f$ is called {\sl ergodic} if the skew product $T_f$ is
topologically ergodic.

We will denote by $\R_T$ the equivalence relation on $X$ generated by $T$:
$\R_T=\{(T^n x,x):x\in X,\:n\in \Z\}$ and by $\wt \R_T$ the equivalence
relation (called the {\sl generic ergodic decomposition} (\cite{We},
\cite{Ke1})) on $X$ defined by: $(x,y)\in \wt \R_T \Leftrightarrow
\ov{T^{\Z} x}=\ov{T^{\Z} y}$.

Let us recall the definition of generic Mackey action associated with a
cocycle (see \cite{GK}). We formulate it here for our situation when the
base transformation group is $\Z$ (acting by $T$) and the cocycle $f:X\to
G$ is continuous. Let $\Omega=(X\times G)/\wt \R_{T_f}$ denote the
factor-space with the factor-topology and $\phi:X\times G\rightarrow
\Omega$ the factor-map. The following properties of $\Omega$ are described
in \cite{GK}: $\Omega$ is a Baire, second countable $T_0$-space (not
necessarily Hausdorff). The map $\phi$ is open, for any meager $S \subset
\Omega$, $\phi^{-1}(S)$ is meager too, and for any meager $\widetilde
\R_{T_f}$-invariant $L\subset X\times G$, $\phi (L)$ is meager. The Borel
structure on $\Omega$ generated by its topology is standard. Let $V(G)$ be
the right translation action on $X\times G$: $V(g)(x,h)=(x,hg^{-1})$.
\begin{definition} (\cite{GK}) The action $W_f(G)$ of the group $G$ on the
space $\Omega$ defined by
$$
W_f(g)\omega = \phi (V(g) y),
$$
where $y\in \phi ^{-1}(\omega)$, $\omega \in \Omega, g \in G,$ is called
the {\bf generic Mackey action} associated with the cocycle $f$.
\end{definition}

It is shown in \cite{GK} that $W_f(G)$ is a continuous action and it is
minimal iff $T$ is. We will denote by $G_{\omega}$, where $\omega\in
\Omega$, the stability group at the point $\omega$: $G_{\omega}=\{g\in G:
\: W_f(g)\omega=\omega\}$. Regardless of the fact that $\Omega$ can be even
not $T_1$-space every $G_{\omega}$ is closed (see \cite{GK}).

From now we will suppose that $T$ is minimal. A cocycle $f$ is called {\sl
generically regular} if, modulo a meager subset of $X$, it is (Borel)
cohomologous to a (generally Borel) ergodic cocycle taking values in a
closed subgroup of $G$ (\cite{GK}). This property has been investigated in
\cite{GK} and it was shown there that the generic regularity of a cocycle
$f$ is equivalent to the essential transitivity of the generic Mackey
action $W_f(G)$ (i.e., modulo a meager subset of $\Omega$, $W_f(G)$ is a
transitive action).

\section{Regularity and essential ranges: a generic approach}

%Throughout this section $G$ will stand for a Polish group (not necessary
%locally compact).

%\subsection{Alternative definitions}

The following definition of the notion of topological essential range comes
from \cite{GH} but we give it here in a more general situation.

Let $G$ be a Polish group, $f:X \to G$ a continuous cocycle of a minimal
Polish system $(X,T)$.

\begin{definition}
The local essential range $E_x(f)$ at the point $x \in X$ is defined by:
$g\in E_x(f)$ if for any open neighborhood $U=U(g)$ and any open
neighborhood $O=O(x)$ there exists $n\neq 0$ such that the set $O\cap
T^{-n} O\cap \{x: f(n,x)\in U\}$ is nonempty.
\end{definition}

As in \cite{GH} we will denote by $P_x(f)$ the set $\{ g\in G: (x,g)
\in\overline{T_f^\Z (x,e)} \}$. Then $P_x(f)\subset E_x(f)$ for every $x\in
X$.

The unit $e$ belong to the essential range $E_x$ iff the cocycle $f$ is
recurrent. We will always consider below recurrent cocycles.

Now let us consider the following family of closed subgroups of $G$:
\begin{equation*}\{G_{\phi(x,e)}\}_{x\in X}\end{equation*}
which are stabilizers of the generic Mackey action associated to $f$ at
points of the form $\phi(x,e)$. One easily sees that it satisfies the
conjugacy equation:
\begin{equation*}
G_{\phi(T^n x,e)} = f(n,x) \cdot G_{\phi(x,e)} \cdot f(n,x)^{-1},
\end{equation*}
for all $x\in X$ and $n\in\Z$.

The following proposition shows that on a dense $G_{\delta}$-subset of $X$
this family coincides with the family of essential ranges $\{E_x\}_{x\in
X}$. Thus, we get an alternative generic definition of the notion of local
essential range given in terms of an associated generic Mackey action. By
the way, the proof presented below is a different proof of the fact that
the topological essential ranges are groups on a comeager subset of $X$ and
it fits for a general Polish situation (cf. \cite[Proposition 1.1]{GH},
where it is assumed $X$ is compact and $G$ is locally compact).

\begin{proposition}\label{ess_r}
There exists an invariant dense $G_{\delta}$-set $X_0\subset X$ such that
for every $x\in X_0$ one has $E_x=P_x=G_{\phi(x,e)}$.
\end{proposition}

\begin{proof} Observe that $G_{\phi(x,e)} \subset P_x$ for
any $x\in X$. Indeed, $g^{-1}\in G_{\phi(x,e)}$ means that $V(g^{-1})\R_f
(x,e)=\R_f (x,e)$ so $(x,g)\in \overline{T_f^\Z (x,e)}$. Thus it suffices
to prove that $E_x\subset G_{\phi(x,e)}$ on some dense $G_{\delta}$-subset.
Let $\Omega=(X\times G)/\wt \R _f$ be the topological factor-space. By
\cite[Proposition 6]{GK} there exists a dense $G_{\delta}$-subset $\Omega
_0\subset \Omega$ such that $\Omega _0$ is a Polish space. Then $Y_0=\phi
^{-1}(\Omega _0)$ is dense $G_{\delta}$ in $X\times G$. By the topological
Fubini theorem there exists a dense $G_{\delta}$-set $X_0\subset X$ with
the property $\{g\in G: (x,g)\in Y_0\}$ is comeager in $G$ for every $x\in
X_0$. We claim that $E_x\subset G_{\phi(x,e)}$ on $X_0$. Fix $x_0\in X_0$
and suppose $g\in E_{x_0}$. Let us denote $A_x=\{g\in G: (x,g)\in Y_0\}$.
It follows from the definition of the essential range that there exist
sequences $\{y_k\}$, $\{n_k\}$ with $y_k\rightarrow x_0$,
$T^{n_k}y_k\rightarrow x_0$ and $f(n_k,y_k)\rightarrow g$. Moreover, one
can choose $\{y_k\}\subset X_0$ as for open $O\subset X$, $U\subset G$ and
$n\in \Z$ the set $O\cap T^{-n}O\cap \{x\in X:f(n,x)\in U\}$ is open, so in
case of it is nonempty it intersects $X_0$. Then $A=\bigcap_{k} A_{y_k}
\cap A_{x_0}$ is comeager in $G$. By Pettis theorem (see \cite[9.9]{Ke2})
$A\cdot A^{-1}=G$, so $g$ can be represented of the form $g=g'' \cdot
{g'}^{-1}$ where $g', g'' \in A$. Then all the points $(y_k,g'), T_f^{n_k}
(y_k,g'), (x_0,g'), (x_0,g'')$ belong to $Y_0$ so their corresponding
images under $\phi$ belong to $\Omega_0$. Since $\phi(y_k,g')\rightarrow
\phi(x_0,g')$, $\phi(T_f^{n_k}(y_k,g'))\rightarrow \phi(x_0,g'')$ and
$\phi(y_k,g')=\phi(T_f^{n_k} (y_k,g'))$ we conclude that
$\phi(x_0,g')=\phi(x_0,g'')$ as $\Omega_0$ is already a Hausdorff space.
This yields $W(g)\phi(x,e)=\phi(x,e)$, i.e. $g\in G_{\phi(x,e)}$.
\end{proof}

\medskip

Thus, the only difference between $G_{\phi(x,e)}$ and $E_x$ can occur on a
meager subset of $X$. But even for these exceptional points we have the
inclusion $G_{\phi(x,e)}\subseteq E_x$. Note also that the properties
$E_x=E_x ^{-1}$, $E_x$ is closed and $E_{T^n x}=f(n,x) \cdot E_x \cdot
f(n,x)^{-1}$ are holds on the whole space $X$.

Assume now that $G$ is locally compact. Then the mapping $x\mapsto
G_{\phi(x,e)}$ from $X$ to the space $(S,\mathfrak F)$ of closed subgroups
of $G$ with the Fell topology is Borel and therefore it is continuous on a
dense $G_{\delta}$-subset of $X$.

\medskip

We recall the definition of regularity for cocycles introduced in
\cite{GH}: let $(X,T)$ be a minimal compact system, then a continuous
cocycle $f:X\to G$ is called {\sl regular} if the skew product $T_f$ admits
a surjective orbit closure closure, i.e. there exists a point $(x_0,g_0)\in
X\times G$ such that $\pi_X \big(\ov{T_f ^{\Z} (x_0,g_0)}\big)=X$. Such a
definition differs from that given in \cite{LeM}, where a stronger
condition must be fulfilled.

\medskip

\begin{proposition}\label{reggen}
Let $f$ be a continuous cocycle over a minimal compact system $(X,T)$ with
values in a locally compact group $G$. Then $f$ is regular if and only if
it is generically regular.
\end{proposition}

\begin{proof} Let $\Omega=(X\times G)/\wt \R _f$. Suppose $f$ is regular. Let
$C$ be a surjective orbit closure so that $C=\overline{T_f^\Z (x,e)}$ for
some $x\in X$. It follows $V(G)C=X\times G$ and it is routine to verify
that then $V(G)\wt \R_f [(x,e)]$ is comeager in $X\times G$. This implies
the orbit $W_f (G)\phi (x,e)$ is comeager in $\Omega$, i.e. the generic
Mackey action $W_f(G)$ is essentially transitive and hence $f$ is
generically regular (\cite[Prop. 17]{GK}).

Conversely, suppose $f$ is generically regular. Then, again by \cite[Prop.
17]{GK}, the action $W_f(G)$ is essentially transitive so there exists
$\omega_0 \in \Omega$ with $W_f(G)\omega_0$ comeager in $\Omega$. This
yields $V(G)\wt \R_f [(x_0,g_0)]$ is comeager in $X\times G$ for some
$(x_0,g_0)\in {\phi}^{-1}(\omega_0)$. Hence $\pi(\wt \R_f [(x_0,g_0)])$ is
comeager in $X$, where $\pi$ denotes the projection $X\times G
\longrightarrow X$. Let $C=\overline{T_f^\Z (x_0,g_0)}$. Then $\pi(C)$ is
$T$-invariant and comeager subset of $X$. Let $\{K_n\}_{n\geq 1}$ be a
countable family of compact subsets of $G$ with $G=\bigcup_n K_n$. Then
each set $F_n=\pi\big((X\times K_n)\cap C\big)$ is compact and their union
$\bigcup_{n\geq 1} F_n=\pi(C)$ is comeager. It follows from Baire's
category theorem that for some $m\in \N$ $F_m$ contains a non-empty open
subset $O$ of $X$. Since $(X,T)$ is a minimal compact system we have
$X=\bigcup_{i=1}^k T^{i} (O)$ for some $k\in \N$ and as $\pi(C)$ is
$T$-invariant we conclude $\pi(C)=X$. The latter means that $f$ is regular.
\end{proof}

It easily follows from the above proposition and the generic definition of
essential range that for a regular cocycle there exists a dense
$G_{\delta}$-subset of $X$ on which all the essential ranges are conjugate
(cf. \cite[Theorem 2.2]{GH}).

\begin{remark} Note that the smoothness of a generic Mackey action would imply
regularity of a cocycle in our situation. Furthemore, let us consider the
case when $G$ is a connected Lie group. Suppose we know that, modulo a
meager subset of $X$, each essential range is an almost connected subgroup.
Then, as the generic Mackey action $W_f(G)$ is ergodic with respect to the
$\sigma$-ideal of meager sets, we are in the assumptions of \cite[Corollary
4.3]{Dani} which imply that on a comeager subset of $X$ all essential
ranges are automorphic in $G$. If, additionally, one of the conditions
$(i)$, $(ii)$, $(iii)$, $(iv)$ of \cite[Corollary 4.4]{Dani} is satisfied
(for instance, $G$ is almost algebraic or $E_x$ is compact for all $x$ from
a comeager subset of $X$) we conclude that all essential ranges are
conjugate on a comeager subset of $X$.
\end{remark}

\begin{corollary}\label{cohomology}
Suppose $f$ is a regular continuous cocycle over a compact minimal system
(X,T) with values in a locally compact group $G$. Then the following is
true:

\begin{enumerate}
\item $f$ is Borel cohomologous to a (Borel) cocycle $\wh f$ which
takes values in a closed subgroup $H$ of $G$ and is ergodic in $H$.
\item If $\wt f$ is a continuous cocycle which is Borel cohomologous to $f$ then $\wt f$ is regular.
\end{enumerate}

\end{corollary}

\begin{proof} The first part follows immediately from \ref{reggen}. The
second part is a consequence of \ref{reggen} and the fact that the generic
Mackey action is invariant under the cohomology equivalence (\cite[Prop.
11]{GK}).
\end{proof}

\begin{remark}
Since the generic Mackey action is an invariant of orbit equivalence for
cocycles (see \cite{GK}) we have that the assertion $(2)$ of the corollary
remains to be true even in case of $f$ is (continuously or generically)
orbit equivalent to $\wt f$.
\end{remark}

\medskip

The next technical result permits us to reduce the question of regularity
to consideration of a factor cocycle in the special situation of
factorization by a compact subgroup. For the proof of it we provide here
the generic arguments which, actually, work also for the case of a general
Polish group $G$.

\begin{proposition}\label{compact}
Let $G$ be a locally compact group, $f:X\to G$ a continuous cocycle over a
minimal compact system $(X,T)$. Suppose $K$ is a compact normal subgroup of
$G$ such that the factor cocycle $f_K=f/K$ is regular. Then $f$ is regular.
\end{proposition}

\begin{proof}
The cocycle $f_K$ is cohomologous to an ergodic cocycle
$\varphi=b(Tx)f_K(x) {b(x)}^{-1}$ with values in a closed subgroup $H$ of
$G/K$. Let $\R_T$ be the equivalence relation on $X$ generated by $T$. It
follows from \cite[Th. 3.3, 4.3]{GKS} that there exists a cocycle $\psi$
defined on $\R_T$ which is orbit equivalent to $\varphi$ with
$\mathrm{Ker}\: \psi$ being an ergodic subrelation of $\R_T$. Let
$s:G/K\rightarrow G$ be a Borel section with $\pi_K \circ s=id$, where
$\pi_K: G\to G/K$ is the projection. Then the values of the cocycle $\wt
\psi=(s\circ b)(Tx)\psi(x){(s\circ b)(x)}^{-1}$ on the ergodic subrelation
$\mathrm{Ker}\: \psi$ belong to $K$. This yields $\wt \psi|_{\mathrm{Ker}\:
\psi}$ is generically regular (\cite[Prop. 20]{GK}) and hence $\wt \psi$
is. Note that $f$ is generically orbit equivalent to $\wt \psi$ to complete
the proof.\end{proof}

\medskip

\section{Minimal rotations}

We turn here to the case when the base dynamical system is a minimal
rotation on a compact monothetic group $X$, i.e. $Tx=ax$, where $\{a^n:\:
n\in \Z\}$ is dense in $X$.

\begin{theorem}\label{discrete}
Let $G$ be a discrete countable group. Suppose $f:X\to G$ is a continuous
recurrent cocycle over a minimal compact group rotation $(X,T)$. Then $f$
is regular. Moreover, $f$ is (continuously) cohomologous to a cocycle $\wh
f$ taking values in a finite subgroup $K\subset G$ and $\wh f$ is ergodic.
\end{theorem}

\begin{proof} Let $\rho$ be the invariant metric on $X$ and $d_0$ a metric
on $G$ defined by: $d_0(g_1, g_2)=1$ if $g_1\neq g_2$, and $d_0(g_1,
g_2)=0$ if $g_1=g_2$. Lets define a metric $d$ on $X\times G$ by:
$$d((x_1,g_1), (x_2,g_2))=\rho(x_1,x_2)+d_0(g_1,g_2)$$

Note that since $f$ is uniformly continuous there exists $\delta >0$ such
that for all $x_1,x_2\in X$ with $\rho(x_1,x_2)<\delta$ and $g_1,g_2\in G$
one has: $d(T_f(x_1, g_1),T_f(x_2, g_2))=d((x_1,g_1), (x_2,g_2))$.

Suppose now that $T_f ^{n_k}(x,g)\rightarrow (y,h)$ when $n_k\rightarrow
\infty$ for some $x,y\in X$, $g,h\in G$. Given any $0<\varepsilon<\delta$
there exists $N>1$ with $\rho(T^{n_k}x, y)<\varepsilon$ and $f(n_k,x)g=h$
for all $k>N$. Then, as $\rho(T^j x,T^{-n_k+j}y)=\rho(T^{n_k} x,y)$ for all
$0\leq j\leq n_k$, we have $$d((x,g),T_f ^{-n_k}(y,h))=d(T_f(x,g), T_f
^{-n_k+1}(y,h))=\ldots=d(T_f ^{n_k}(x,g), (y,h))<\varepsilon$$

It follows $T_f ^{-n_k}(y,h)\rightarrow (x,g)$. The latter yields that
every orbit closure $\overline{T_f^\Z (x,g)}$, ($(x,g)\in X\times G$) is
minimal under $T_f$. Since $f$ is recurrent one may assume without loss of
generality that there exists a positive sequence $\{m_k\}_{k>1}\in \Z ^{+}$
such that $T_f ^{m_k}(x,g)\rightarrow (x,g)$ for some $(x,g)\in X\times G$.
By the same argument as above $T_f ^{-m_k}(x,g)\rightarrow (x,g)$ so the
point $(x,g)$ is recurrent in the terminology of \cite{GoHe} (i.e. $(x,g)$
belongs to both positive and negative orbit closures of $(x,g)$). This
together with minimality of $\big(\overline{T_f^\Z (x,g)}, T_f\big)$
implies $\overline{T_f^\Z (x,g)}$ is compact (\cite[Theorem 7.05]{GoHe}),
so $f$ is regular (\cite{GH}). By virtue of \cite[Proposition 2.1]{LeM}
there exists a compact subgroup $K$ (which equals $E_x$ in our case) of $G$
and a continuous map $\gamma : X\to G/K$ such that $\gamma (Tx)=f(x)\gamma
(x)$ for all $x\in X$. Let $s:G/K\to G$ be a section (so that $\pi_K \circ
s=id$, where $\pi_K: G\to G/K$ is the projection). Put $b=s\circ\gamma$ and
$\wh f = {b(Tx)}^{-1}f(x)b(x)$. Then one verifies that $\wh f$ satisfies
the conditions of the theorem.\end{proof}

\begin{corollary} Let $G$ be a discrete countable group without finite subgroups.
Let $f$ be any continuous cocycle over a minimal compact group rotation
$(X,T)$ with values in $G$. Then either $f$ is a coboundary or the skew
product action is an action with only discrete orbits.
\end{corollary}

\begin{corollary} Let $(X,T)$ be a minimal compact group rotation.
Suppose $G$ is a locally compact s.c. group such that there exists a
compact normal subgroup $K\subset G$ with $G/K$ discrete (in particular,
when $G$ is a Lie group with the compact identity component or a totally
disconnected abelian group). %, or a $p$-adic Lie group).
Then every recurrent cocycle $f:X\to G$ is regular and is (continuously)
cohomologous to an ergodic cocycle taking values in a compact subgroup of
$G$.
\end{corollary}

\begin{proof} The proof repeats the arguments of \ref{discrete}.\end{proof}

\begin{corollary}\label{alm_conn} Let $(X,T)$ be a minimal compact group rotation.
Suppose $G$ is a Lie group or a locally compact s.c. abelian group. Then
every recurrent cocycle $f:X\to G$ is (continuously) cohomologous to a
cocycle taking values in an almost connected (closed) subgroup $H$ of $G$.
In particular, there are no continuous ergodic cocycles over $(X,T)$ taking
values in a Lie group with infinitely many connected components.
\end{corollary}

\begin{proof} Let $G^0$ be the identity component of $G$ and $f_0:X \to
G/G^0$ the factor cocycle. Suppose, for a start, $G$ is a Lie group. By
\ref{discrete} we have $f_0=b(Tx)\wt f_0 {b(x)}^{-1}$, where $\wt f_0$ is a
cocycle taking values in a finite subgroup $F$ of $G/G_0$ and
$b:X\rightarrow G/G_0$ is continuous. Let $s: G/G_0\rightarrow G$ be the
section. Lets define $\wt f(x)=(s\circ b)(Tx)f(x){(s\circ b)(x)}^{-1}$.
Then evidently $\wh f$ satisfies the conditions of our assertion.

Now suppose $G$ is a locally compact abelian group. Then $G/G^0$ is totally
disconnected so there exists a compact subgroup $K\subset G/G^0$ with
$(G/G^0)/K$ being discrete. By virtue of \ref{discrete} one may assume that
the cocycle $f_K$ defined by $f_K=f_0/K$ takes all its values in a finite
subgroup of $(G/G^0)/K$. Thus $f_0$ takes values in a compact subgroup of
$G/G^0$ and the similar argument as above implies our assertion.
\end{proof}

We complete with the following result on regularity of cocycles over an
arbitrary minimal rotation on a compact group with values in an arbitrary
locally compact abelian group. In \cite{Me} the similar result was proved
for the case when $G$ does not contain compact subgroups, however it was
used there a little different notion of regularity (see \cite[2.7]{Me}).
The case of locally connected compact group $X$ is covered by \cite[Theorem
4.1]{GH}.

\begin{theorem}
Let $(X,T)$ be a minimal compact group rotation, $G$ a locally compact s.c.
abelian group. Then every recurrent cocycle $f:X\to G$ is regular.
\end{theorem}

\begin{proof}
There exists a compact subgroup $K\subset G$ with $G/K$ being a Lie group.
So, in view of \ref{compact}, it suffices to prove our assertion for the
case $G=\mR ^n \times \T ^k \times D$, where $D$ is discrete. The
application of \ref{alm_conn} allows us to suppose that $D$ is finite and,
again, by \ref{compact} the situation is reduced to the case $G=\mR ^n$.
The latter was shown in \cite{Me} but we give here a slightly different
argument to illustrate our approach. Recall that in the abelian case all
the essential ranges are the same and equal to a closed subgroup $E(f)$ of
$G$ (\cite{LeM, GH}). So $E(f)\cong\mR ^k \times \Z ^m$. If $k\neq 0$ lets
consider the factor cocycle $\wt f:X\to G/E(f)\cong \mR ^{n-m-k}\times \T
^m$. Evidently, the existence of a surjective orbit closure for $\wt f$
would imply the regularity of $f$. So, arguing as before, one may assume
without loss of generality that $\wt f$ is a cocycle with values in $\mR
^{n-m-k}$. As we know that any recurrent $\mR$-valued cocycle is regular
(\cite[Theorem 1]{LeM}) the application of a standard inductive argument
completes the proof.
\end{proof}


\begin{thebibliography}{99}

\bibitem[At]{Atkinson} G. Atkinson,  {\it A class of transitive cylinder transformations},
J. London Math. Soc. (2) {\bf 17}, (1978), no. 2, 263--270.

\bibitem[Da]{Dani} S.G. Dani, {\it On conjugacy classes of closed subgroups and stabilizers
of Borel actions of Lie groups}, Ergod. Th. \& Dynam. Sys., {\bf 22},
(2002), 1697--1714.

\bibitem[GKS]{GKS} V. Golodets, V. Kulagin, and S. Sinel'shchikov, {\it Orbit
properties of pseudo-homeomorphism groups and their cocycles}, London Math.
Soc. Lecture Note Series, {\bf 277}, (2000), Cambridge Univ. Press,
Cambridge, 211 -- 229.

\bibitem[GK]{GK} V. Golodets and V. Kulagin, {\it Weak equivalence of cocycles
and Mackey action in generic dynamics}, Qualitative Theory of Dynamical
Systems, {\bf 4}, No 1, (2003), 39 -- 57.

\bibitem[GoHe]{GoHe} Gottschalk, W.H.; Hedlund, G. A. {\it Topological dynamics},
American Mathematical Society Colloquium Publications, Vol. {\bf 36},
Providence, R. I., 1955.

\bibitem[GH]{GH} G. Greschonig, U. Hab\"ock, {\it Nilpotent extensions of minimal homeomorphisms},
Ergod. Th. \& Dynam. Sys., {\bf 25}, (2005), 1--17

\bibitem[Ke00]{Ke1} A. Kechris, {\it Descriptive Dynamics}, London Math. Soc.
Lecture Note Series, {\bf 277}, (2000). Cambridge Univ. Press, Cambridge,
231 -- 258.

\bibitem[Ke95]{Ke2} A. Kechris, {\it Classical Descriptive Set Theory}, Graduate
Texts in Mathematics, {\bf 156}. Springer-Verlag, New-York, 1995.

\bibitem[LM]{LeM} M. Lema\'nczyk, M. Mentzen,  {\it Topological ergodicity of real cocycles
over minimal rotations}, Monatsh. Math. {\bf 134} (2002), no. 3, 227--246.

\bibitem[Me]{Me} M. Mentzen,  {\it On groups of essential values of topological cyclinder
cocycles over minimal rotations}, Colloq. Math. {\bf 95} (2003), no. 2,
241--253.

\bibitem[We]{We} B. Weiss, {\it A survey of generic dynamics.}, London Math.
Soc. Lecture Note Series, {\bf 277}, (2000), Cambridge Univ. Press,
Cambridge, 273 -- 291.

\end{thebibliography}
\end{document}